\documentclass[12pt,reqno]{article}

\usepackage[usenames]{color}
\usepackage{amssymb}
\usepackage{graphicx}
\usepackage{amscd}
\usepackage{amsmath}
\usepackage[colorlinks=true,
linkcolor=webgreen,
filecolor=webbrown,
citecolor=webgreen]{hyperref}

\definecolor{webgreen}{rgb}{0,.5,0}
\definecolor{webbrown}{rgb}{.6,0,0}

\usepackage{color}
\usepackage{fullpage}
\usepackage{float}
\usepackage[mathcal]{euscript}
\usepackage{graphics,amssymb}
\usepackage{amsthm}
\usepackage{amsfonts}
\usepackage{latexsym}
\usepackage{epsf}
\usepackage{subfigure}
\usepackage{booktabs}
\usepackage{array}
\usepackage{multicol}

\newcommand{\seqnum}[1]{\href{http://oeis.org/#1}{\underline{#1}}}

\begin{document}

\theoremstyle{plain}
\newtheorem{theorem}{Theorem}
\newtheorem{corollary}[theorem]{Corollary}
\newtheorem{lemma}[theorem]{Lemma}
\newtheorem{proposition}[theorem]{Proposition}

\theoremstyle{definition}
\newtheorem{definition}[theorem]{Definition}
\newtheorem{notation}[theorem]{Notation}
\newtheorem{example}[theorem]{Example}
\newtheorem{conjecture}[theorem]{Conjecture}

\theoremstyle{remark}
\newtheorem{remark}[theorem]{Remark}

\begin{center}
\vskip 1cm{\LARGE\bf A state enumeration of the foil knot\\
\vskip .11in
}

\vskip 1cm
\large
Franck Ramaharo and Fanja Rakotondrajao\\
Universit\'e d'Antananarivo\\
D\'epartement de Math\'ematiques et Informatique\\
101 Antananarivo,	Madagascar\\
\href{mailto:franck.ramaharo@gmail.com}{\tt franck.ramaharo@gmail.com} \\
\href{mailto:frakoton@yahoo.fr}{\tt frakoton@yahoo.fr}
\end{center}
\begin{center}
\today
\end{center}
\begin{abstract}
We split the crossings of the foil knot and enumerate the resulting states with a generating polynomial. Unexpectedly, the number of such states which consist of two components are given by the lazy caterer's sequence. This sequence describes the maximum number of planar regions that is obtained with a given number of straight lines. We then establish a bijection between this partition of the plane and the concerned foil splits sequence.
\end{abstract}

\section{Introduction}

Knot theory is a branch of algebraic topology which deals with the properties of figures that remain unaltered under continuous deformations. A mathematical knot is defined as a closed, non-self-intersecting curve that is embedded in three dimensions. Enumerating knots is an active field of research in knot theory. We attribute the most notable one to Tait \cite{Tait}, who attempted to enumerate knots based on the minimum number of crossings \cite{stoimenow}. Other enumeration tools include Gauss codes \cite{Gausscode}, matrix models \cite{Matrixmodel}, tricolorability \cite{tricolorability} and knot polynomials \cite{Homfly,Kauffman}. In this paper, we compute the ``states'' \cite[p.\ 25]{Kauffman1} of a particular subset of the knot family with an easier combinatorial approach to the Kaufman state-sum model. The computation is based on a particular representation of the knot and we define a generating polynomial whose coefficients enumerate a property of the Kauffman states.

Based on the previous definition of the mathematical knots, the trivial one, called the \textit{unknot}, is simply a closed curve that can be continuously deformed into a circle without self-cuttings. Besides, one of the simplest family of knots is the \emph{foil} family which is a subset of the \emph{torus knot} family. A torus knot \cite[p.\ 107]{Adams} is a special kind of knot that can be drawn on the surface of a torus without self-intersections. A foil knot is obtained by winding $ n $ times around a circle in the interior of the torus, and $ 2 $ times around its axis of rotational symmetry. Some well-known representatives of the foil family are pictured in \hyperref[fig:foil3d]{Figure \ref*{fig:foil3d}}.

\begin{figure}
\centering
\includegraphics[width=0.15\linewidth]{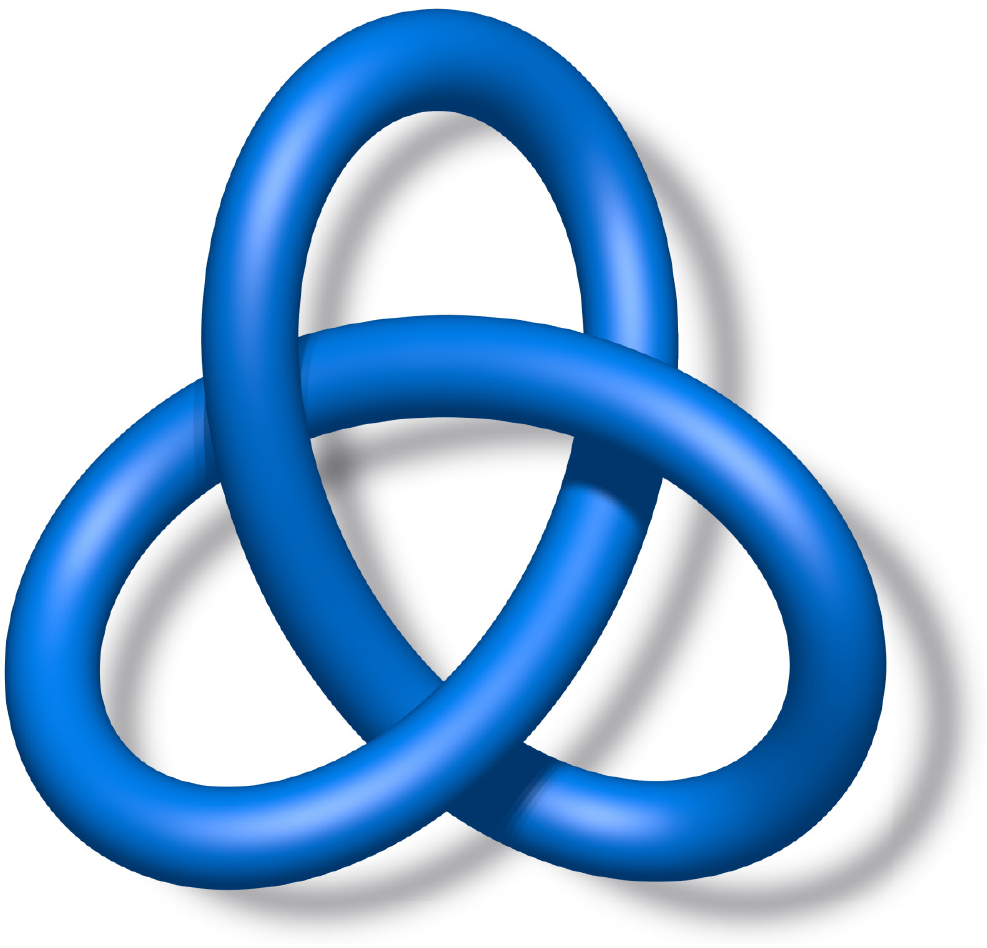}
\includegraphics[width=0.15\linewidth]{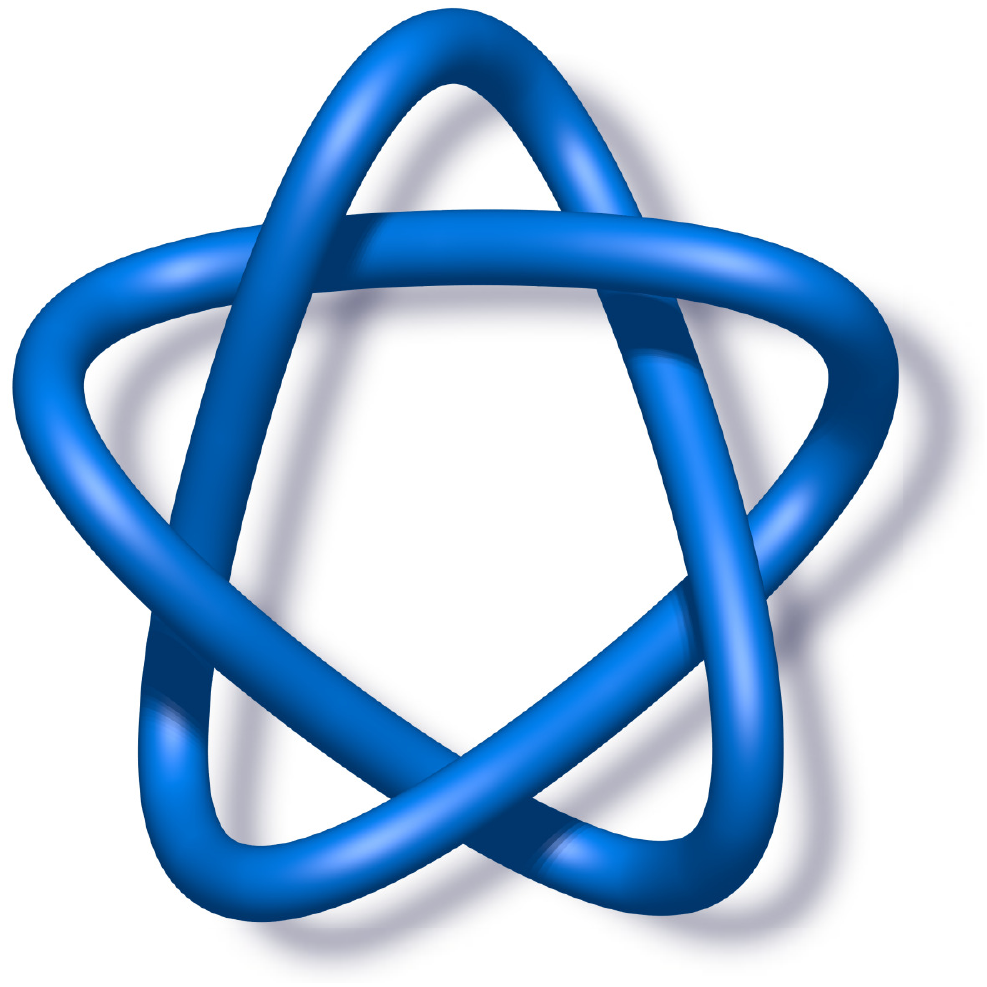}
\includegraphics[width=0.15\linewidth]{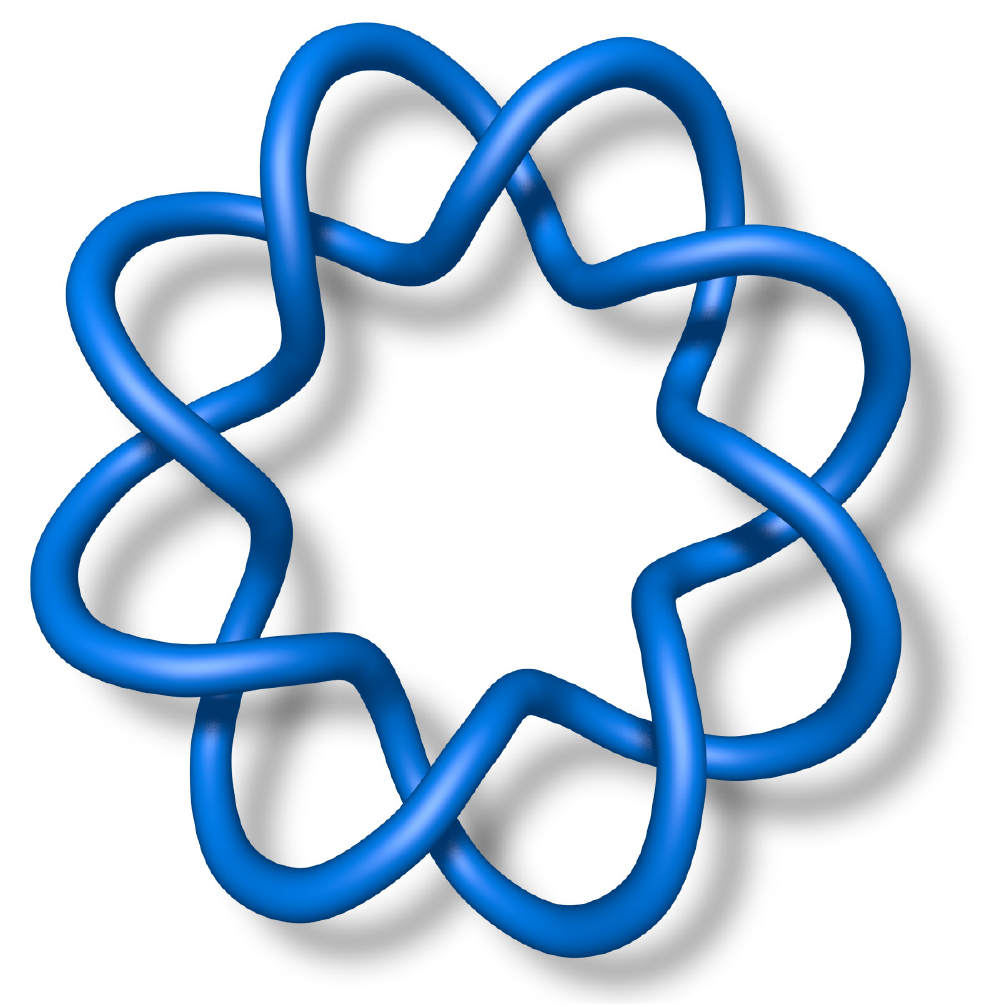}
\caption{Fom the left to the right: the trefoil, the pentafoil and the octafoil knots.}
\label{fig:foil3d}
\end{figure}

\begin{definition}
	Let $ n $ be a nonnegative integer. An $ n $-foil is a representative of the foil family whose shadow diagram consists of $ n $ crossings. We let $ F_n $ denote the $ n $-foil.
\end{definition}

It is generally more convenient to deal with knots with the so-called \textit{knot diagram}. However, for the sake of simplicity, we will rather work with the \emph{shadow diagram} \cite{Denton} which is the regular projection of the three-dimensional knot onto the plane. For example, we see in \hyperref[fig:foil-shadow]{Figure \ref*{fig:foil-shadow}} the shadows of the trefoil, the pentafoil and the octafoil knots.

\begin{figure}
\centering
\includegraphics[width=0.15\linewidth]{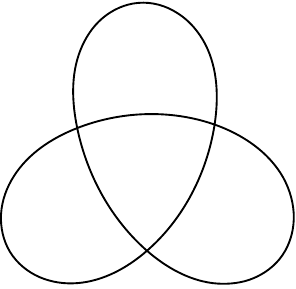}
\includegraphics[width=0.15\linewidth]{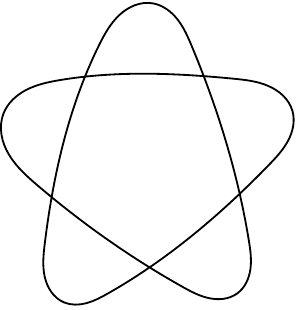}
\includegraphics[width=0.15\linewidth]{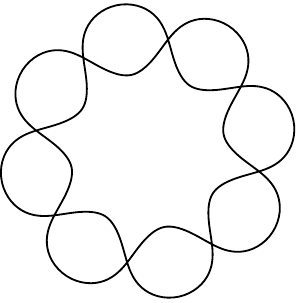}
\caption{Shadow diagrams of the trefoil, the pentafoil and the octafoil knots.}
\label{fig:foil-shadow}
\end{figure}

We abusively make no distinction between a knot, its diagram or its shadow throughout this paper. We point out that according to the Jordan curve theorem \cite[p.\ 7]{Jordan}, the shadow diagram can be checkerboard colored, i.e., we may fill a diagram with two colors such that any two faces that share a boundary arc have opposite colors (see \hyperref[fig:jordan]{Figure \ref*{fig:jordan}}).

\begin{figure}
\centering
\includegraphics[width=0.4\linewidth]{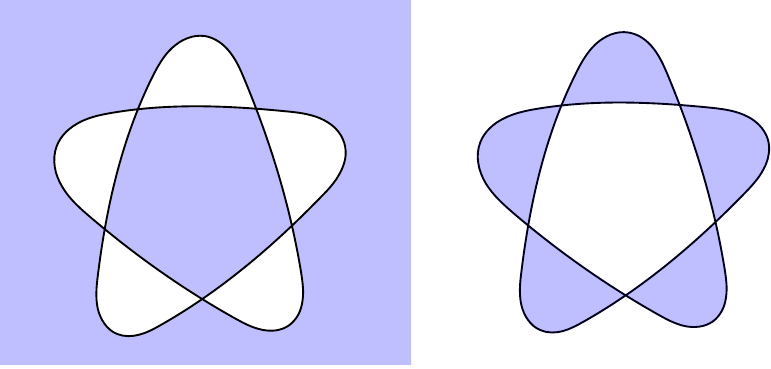}
\caption{The checkerboard colorings of the pentafoil.}
\label{fig:jordan}
\end{figure}

We say that two areas are \emph{neighbors} if their common borderline is an arc, and \emph{opposite} if their common boundary is a double point. We say that the area which is opposed to the unbounded one, as well as each of the related opposite areas, is the \textit{$ A $-region}. A neighbor of the $ A $-region is then referred to as the \textit{$ B $-region}. From this point of view, we shall apply a split operation at a given crossing.

\begin{definition}
We may \emph{split} (or \emph{smooth}) a crossing in two ways such that either of the $ A $-regions or the $ B $-regions are respectively joined into the so-called \textit{$ A $-channel} and the\textit{ $ B $-channel}. We say that a split is of \textit{type $ A $} or an \textit{$ A $-split} (correspondingly\ of \textit{type $ B $} or a \textit{$ B $-split}) if it opens the $ A $-channel (correspondingly the $ B $-channel). The split process is illustrated in \hyperref[fig:state-of-crossing]{Figure \ref*{fig:state-of-crossing}}.

\begin{figure}
\centering		
\includegraphics[width=0.3\linewidth]{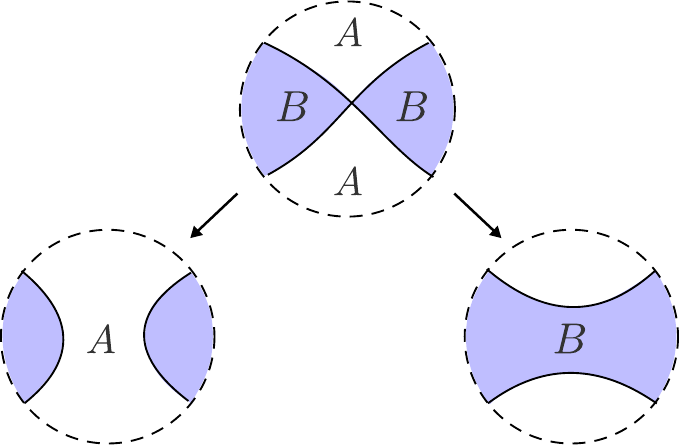}
\caption{The splits of a crossing.}
\label{fig:state-of-crossing}		
\end{figure}

A \textit{state} for the diagram $ D $ is a diagram obtained by opening the $ A $ channel or the $ B $ channel at each crossing. The result is a union of mutually disjoint Jordan curves (or components). We let a \textit{$ k $-state} denote a state which consists of the union of $ k $ Jordan curves. Notice that the number of states of a  diagram with $ n $ crossings is then $ 2^n $. For example, \hyperref[fig:trefoil-decomposition]{Figure \ref*{fig:trefoil-decomposition}} displays the $ 2^3 $ states of the trefoil $ F_3 $ arranged in a binary tree. We then associate the enumeration of these states with the so-called \textit{states summation} \cite[p.\ 23] {StatesSummation}.

\begin{figure}
\centering
\includegraphics[width=0.8\linewidth]{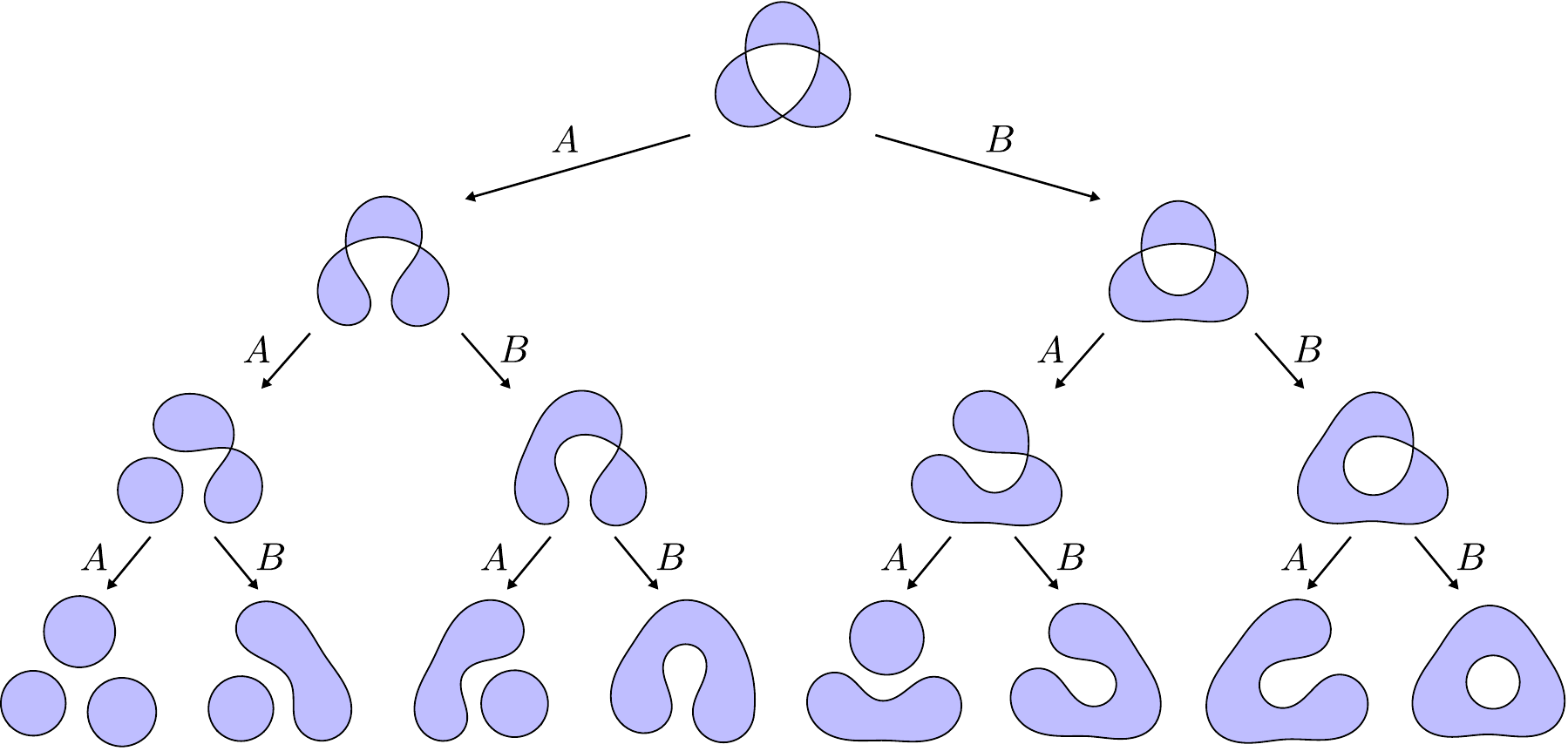}
\caption{The states for the trefoil diagram.}
\label{fig:trefoil-decomposition}
\end{figure}
\end{definition}

\begin{definition}
Let $ S $ be a state of a diagram $ D $, and let $ |S| $ denote the number of components of the state $ S $. The generating polynomial associated with the diagram $ D $ is the univariate polynomial defined by
\begin{equation*}
D(x)=\sum\limits_{S}^{}x^{|S|}
\end{equation*}
where the summation is taken over all states for $ D $.
\end{definition}

Accordingly, the trefoil $ F_3 $ of \hyperref[fig:trefoil-decomposition]{Figure \ref*{fig:trefoil-decomposition}} has the following generating polynomial:
\begin{equation}\label{eq:F3}
F_3(x)=x^3+4x^2+3x.
\end{equation}

\begin{remark}
We may continuously deform a planar diagram such that we obtain another representation which is \textit{planar isotopic} \cite[p.\ 12]{Adams} to the former diagram, and the crossings remaining unaltered (see \hyperref[fig:12foil]{Figure \ref*{fig:12foil}}). We can then write the generating polynomial of the $ 2 $-foil and the $ 1 $-foil with the help of \hyperref[fig:trefoil-decomposition]{Figure \ref*{fig:trefoil-decomposition}}. They are respectively
\begin{equation}\label{eq:F1}
F_1(x)=x^2+x
\end{equation}
and
\begin{equation}\label{eq:F2}
F_2(x)=2x^2+2x.
\end{equation}

\begin{figure}
\centering
\hspace*{\fill}
\subfigure[$ 1 $-foil\label{sub:1foil}]{\includegraphics[width=0.3\linewidth]{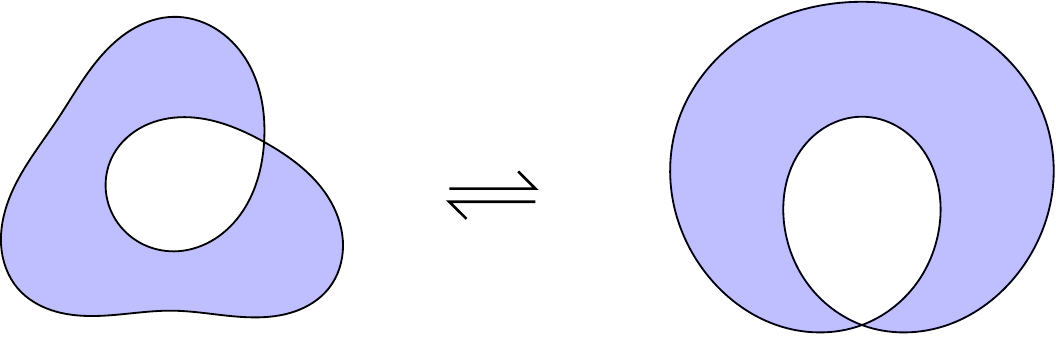}}\hfill%
\subfigure[$ 2 $-foil\label{sub:2foil}]{\includegraphics[width=0.3\linewidth]{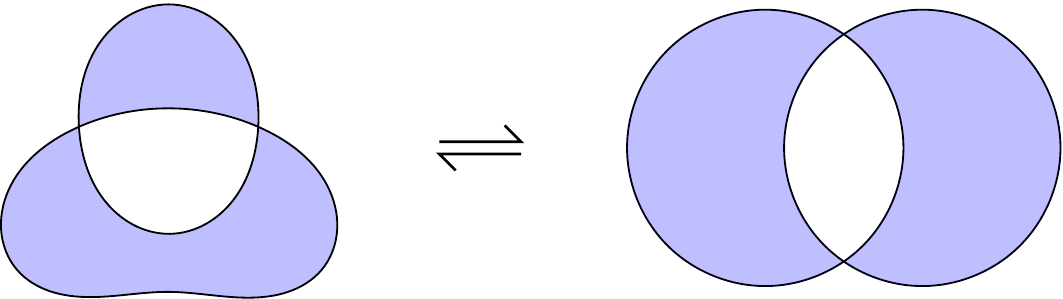}}
\hspace*{\fill}
\caption{Two equivalent representations of the $ 1 $-foil and the $ 2 $-foil.}
\label{fig:12foil}
\end{figure}
\end{remark}

Let the polynomial $ F_n(x) $ be expressed as follows:
\begin{equation*}
F_n(x):=\displaystyle{\sum_{k\geq 0}f_{n,k}x^k},\ \mbox{where}\ f_{n,k}:=\#\{S\mid S\mbox{ is a state of }F_n\ and\ |S|=k\}.
\end{equation*}
We are particularly interested in the coefficients of the monomial $ f_{n,2}x^2 $. The expressions \eqref{eq:F3}, \eqref{eq:F1} and \eqref{eq:F2} give us some of the early values. In order to complete the sequence $ (f_{n,2})_{n\geq0} $, we compute the generating polynomial of the $ n$-foil knot in \hyperref[sec:foilknot]{section \ref*{sec:foilknot}} and show that the concerned sequence is actually the lazy caterer's sequence. Such sequence gives the maximal numbers a plane may be divided from a given number of straight lines. In \hyperref[sec:foilknot]{section \ref*{sec:particularset}}, we introduce a particular configuration of straight lines on the plane. This configuration allow us to encode all the present planar regions in \hyperref[sec:planeencoding]{section \ref{sec:planeencoding}}. We encode as well the states diagrams of the $ n $-foil in \hyperref[sec:states]{section \ref*{sec:states}}. These encodings serve as key ingredients that allow us to construct a bijection between the planar regions and the state diagrams in \hyperref[sec:bijection]{section \ref*{sec:bijection}}. 

\section{The generating polynomial of the foil knot}\label{sec:foilknot}
The present section is concerned with the computation of the generating polynomial of the $ n $-foil. Let us begin by presenting another simplest family of knot.

\begin{definition}
We recall that the trivial knot is the \textit{unknot}. If we twist the unknot around itself and project the result into the plane, then we obtain a \textit{twisted loop}. If such knot has $ n $ half-twists, then we call it an \textit{$ n $-twist loop}, and we let $ T_n $ denote this knot.
\end{definition}

The unknot and a $ 5$-twist loop are illustrated in \hyperref[fig:5twistloop]{Figure \ref*{fig:5twistloop}}.

\begin{figure}[H]
\centering
\includegraphics[width=0.5\linewidth]{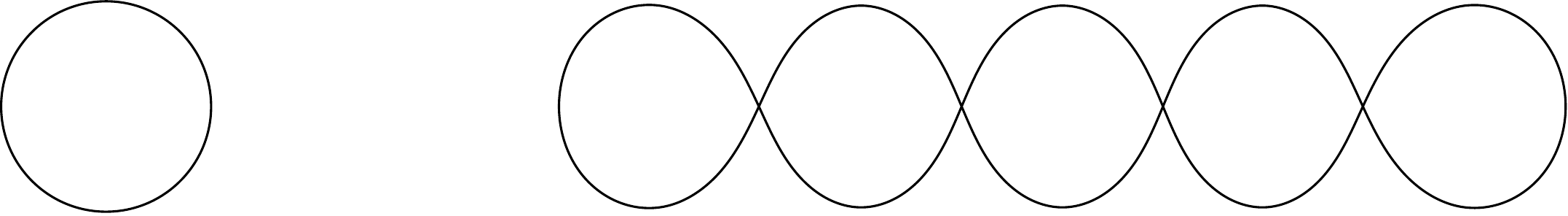}
\caption{The unknot ($ 0 $ half-twist) and a $ 5$-twist loop.}
\label{fig:5twistloop}
\end{figure}

We point out that the $ n$-foil may be obtained from the $ n$-twist loop by removing a small arc from the extremity and gluing the strands together without introducing a new crossing. For example, we see in \hyperref[fig:transf-twist-foil]{Figure \ref*{fig:transf-twist-foil}} how a $ 5 $-foil is obtained from the previous $ 5$-twist loop.

\begin{figure}[H]
\centering
\includegraphics[width=.7\linewidth]{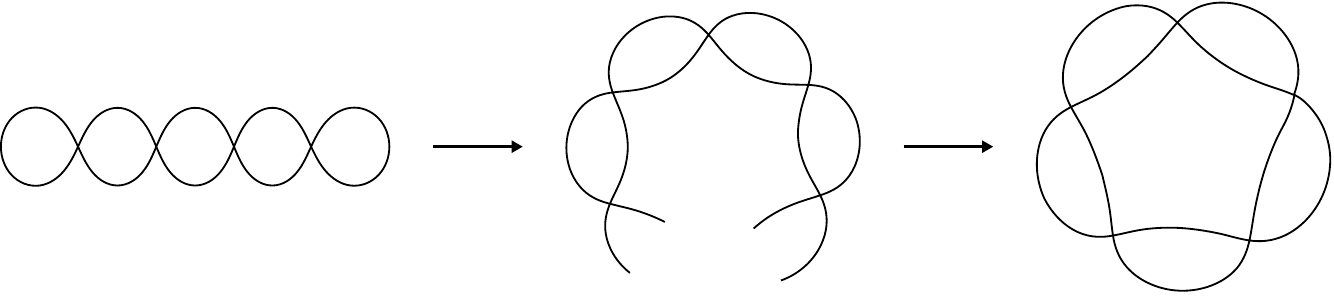}
\caption{From a $ 5$-twist loop to a $ 5$-foil.}
\label{fig:transf-twist-foil}
\end{figure}

Using the same process, we define the diagram of the $ 0$-foil as in \hyperref[fig:0-foil]{Figure \ref*{fig:0-foil}}.

\begin{figure}[H]
\centering
\includegraphics[width=0.8\linewidth]{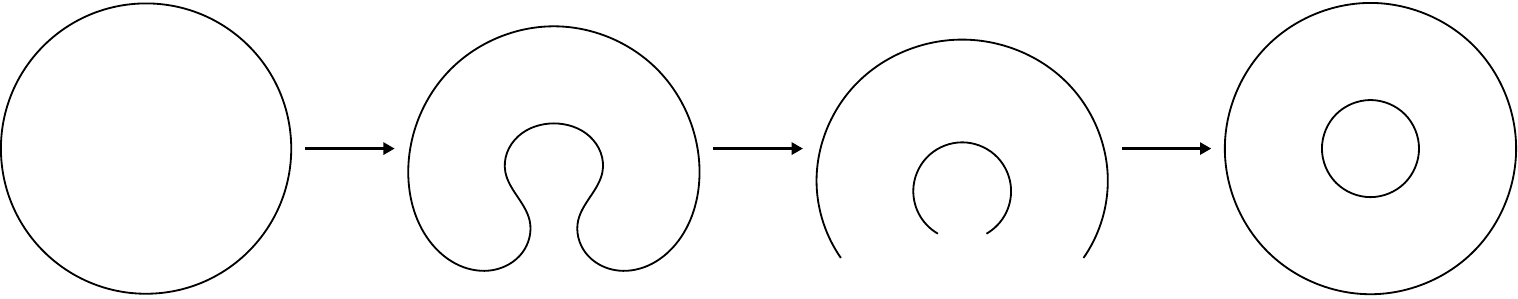}
\caption{The construction of the 0-foil.}
\label{fig:0-foil}
\end{figure}

These definitions already allow us to write the generating polynomial of $ T_0 $ and $ F_0 $. They are respectively given by
\begin{equation*}
T_0(x)=x
\end{equation*}
and
\begin{equation*}
F_0(x)=x^2.
\end{equation*}

Let us next establish the generating polynomial of the $ n $-foil.

\begin{proposition}
The generating polynomial of the $ n $-foil verifies
\begin{equation}\label{eq:foiltwist}
F_n(x)=T_{n-1}(x)+F_{n-1}(x),\ n\geq1, 
\end{equation}
where $ T_{n-1}(x) $ denotes the generating polynomial of the $ (n-1) $-twist loop.
\end{proposition}

\begin{proof}
Let us split a crossing of the $ n $-foil as illustrated in \hyperref[fig:nfoilstate]{Figure \ref*{fig:nfoilstate}}. The split either results in a $ (n-1 )$-foil or a $ (n-1) $-twist loop.
	
\begin{figure}[H]
\centering
\includegraphics[width=0.55\linewidth]{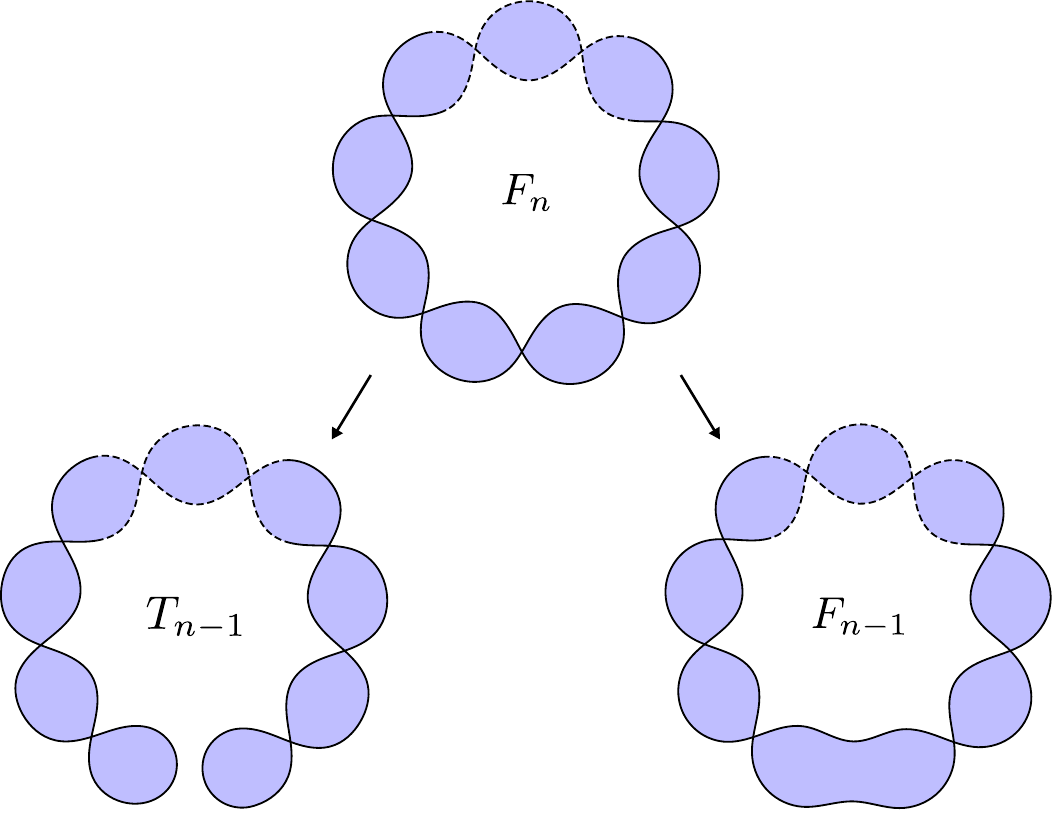}
\caption{The splits of an $ n $-foil crossing.}
\label{fig:nfoilstate}
\end{figure}

Therefore, the generating polynomial is given by
\begin{equation*}
F_n(x)=T_{n-1}(x)+F_{n-1}(x).
\end{equation*}
\end{proof}

Notice that this relationship invokes the generating polynomial of the $ n $-twist loop.

\begin{proposition}
The generating polynomial of the $ n $-twist loop is given by
\begin{equation}\label{eq:twloop}
T_n(x)=x(x+1)^n,\ n\geq 0.
\end{equation}
\end{proposition}

\begin{proof}
Let $ n \geq 1$. The key point is to split the leftmost crossing of the $ n $-twist loop. Either we obtain a disjoint union of an $ (n-1) $-twist loop and the unknot, or we obtain an $ (n-1) $-twist loop.

\begin{figure}[H]
\centering
\includegraphics[width=0.8\linewidth]{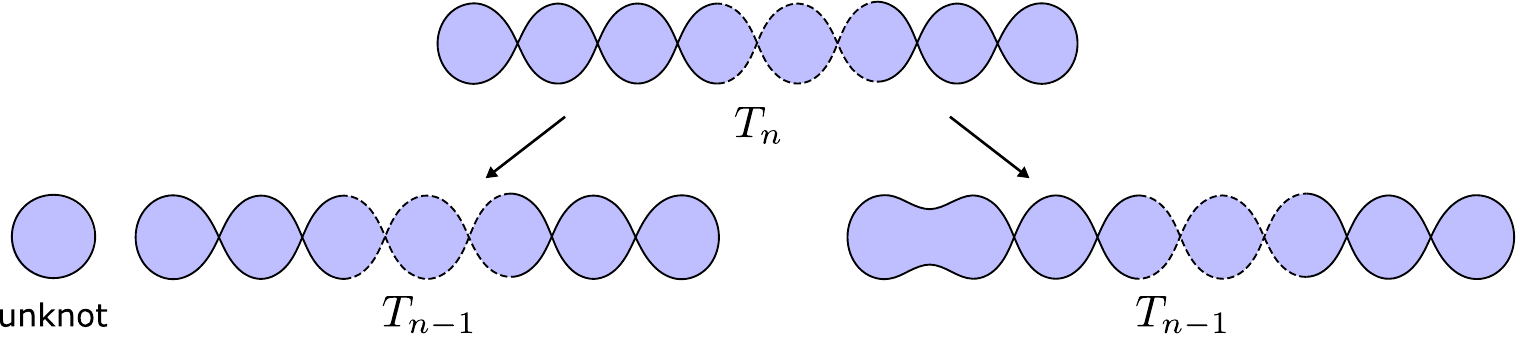}
\caption{The splits of an $ n $-twist loop crossing.}
\label{fig:twistloop}
\end{figure}

Therefore we write the corresponding generating polynomial as follows:
\begin{equation}
T_n(x)=xT_{n-1}(x)+T_{n-1}(x).
\end{equation}
Taking into account the expression $$ T_0(x)=x, $$ we have
\begin{equation}
T_n(x)=x(x+1)^n.
\end{equation}
\end{proof}

We let $ t_{n,k} $ denote the coefficients of the polynomial $ T_n(x) $, and we give the corresponding values for $ 0\leq n\leq 6 $ and $ 0\leq k\leq 7 $ in \hyperref[tab:twistloop]{Table \ref*{tab:twistloop}}.

\begin{table}[H]
\[
\begin{array}{c|rrrrrrrr}
n\ \backslash\ k	&0	&1	&2	&3	&4	&5	&6	&7\\
\midrule
0	&0	&1	&	&	&	&	&	&\\
1	&0	&1	&1	&	&	&	&	&\\
2	&0	&1	&2	&1	&	&	&	&\\
3	&0	&1	&3	&3	&1	&	&	&\\
4	&0	&1	&4	&6	&4	&1	&	&\\
5	&0	&1	&5	&10	&10	&5	&1	&\\
6	&0	&1	&6	&15	&20	&15	&6	&1
\end{array}
\]
\caption{Values of $ t_{n,k} $ for $ 0\leq n\leq 6 $ and $ 0\leq k\leq 7 $.}
\label{tab:twistloop}
\end{table}

The coefficients in \hyperref[tab:twistloop]{Table \ref*{tab:twistloop}} are those of the binomial expansion with a horizontal shift. The values of $ t_{n,k} $ for $ n\geq k-1 $ are given by $ \binom{n}{k-1} $. Such expression exactly matches  the property of the enumeration of the twist loop states. In order to obtain a $ k $-state, we must choose $ k-1 $ crossings at which we apply an $ A $ split. At this point, our twist loop is divided into $ k $ ``smaller'' pieces of twist loops. We then apply a $ B $-split at each of the remaining crossings. Consider the following figure.

\begin{figure}[H]
\centering
\includegraphics[width=0.65\linewidth]{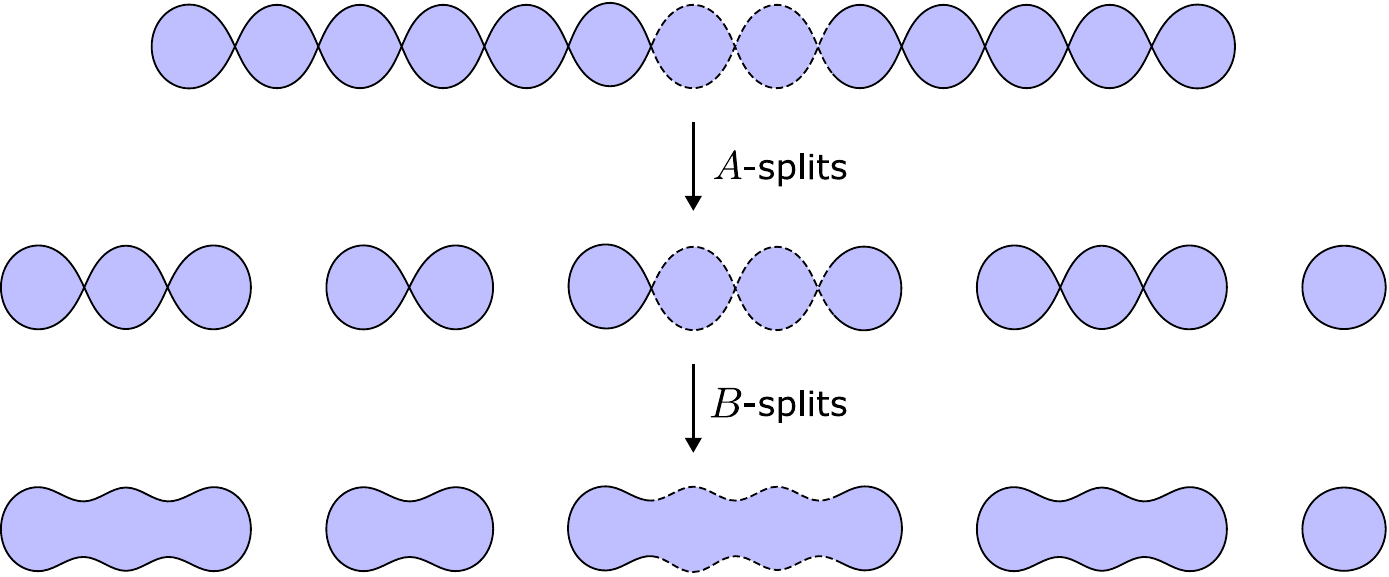}
\caption{A series of $ k $ $ A $-splits and $ n-k $ $ B $-splits produces $ k+1 $ components.}
\label{fig:ABsplittwistloop}
\end{figure}

Let us pay a special attention to the following integer sequences.

\begin{itemize}
\item The zero sequence \seqnum{A000004}:
\[	0, 0, 0, 0, 0, 0, 0, 0, 0, 0, 0, 0, 0, 0, 0, 0, 0, 0, 0, 0, \ldots\]
Regardless of the number of crossings, a state consists at least of one component. Such constraint holds for any knot diagram.

\item The simplest sequence of positive numbers, the all $ 1 $'s sequence \seqnum{A000012}:
\[	1, 1, 1, 1, 1, 1, 1, 1, 1, 1, 1, 1, 1, 1, 1, 1, 1, 1, 1, 1, \ldots\]
Let us distinguish the second column $ (t_{n,n+1})_{n\geq 0} $ and the diagonal $ (t_{n,1})_{n\geq 0} $.
\begin{itemize}
\item $ (t_{n,n+1})_{n\geq 0} $: we obtain one $ (n+1) $-state in a unique way, that is by applying a $ A $-split at every crossing.

\begin{figure}[H]
\centering
\includegraphics[width=0.65
\linewidth]{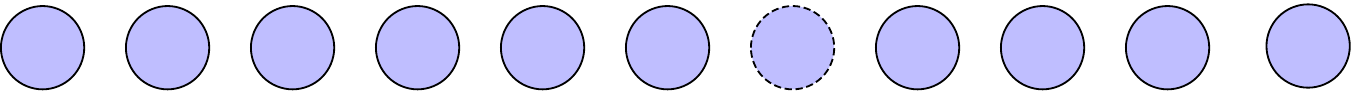}
\caption{A ``full'' $ A $-split results to a $(n+1) $-state.}
\label{fig:Asplittwistloop}
\end{figure}
		
\item $ (t_{n,1})_{n\geq 0} $: conversely, we have one possibility of obtaining a one component state by only applying a $ B $-split at every crossing.

\begin{figure}[H]
\centering
\includegraphics[width=0.5
\linewidth]{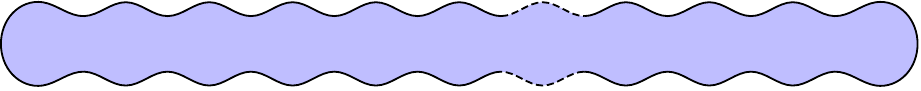}
\caption{A ``full'' $ A $-split results to a $ 1 $-state.}
\label{fig:Bsplittwistloop}
\end{figure}
\end{itemize}

\item The sequence of nonnegative integers \seqnum{A001477}:
\[0,1,2,3,4,5,6,7,8,9,10,11,12,13,14,15,16,17,\ldots\]
Consider the diagonal $ (t_{n,n})_{n\geq 0} $ whose general term is $ t_{n,n}=\binom{n}{n-1}=n $. We choose $ n-1 $ crossings at which we apply a $ B $-split. At the one crossing left, we apply an $ A $-split so that we now end up to $ n $ components.
\end{itemize}

Finally, we may express the polynomial $ F_n(x) $ as the following closed form.

\begin{corollary}
The generating polynomial of the $ n $-foil knot is given by
\begin{equation}\label{eq:foil}
F_n(x)=(x+1)^n+x^2-1,\ n\geq 0 .
\end{equation}
\end{corollary}

\begin{proof}
Taking into consideration the relation (\ref{eq:twloop}) and the fact that $ F_0(x)=x^2 $, we have
\begin{align*}
F_n(x)&=\sum_{k=0}^{n-1}T_k(x)+F_0(x)\\
&=x\sum_{k=0}^{n-1}(x+1)^k+x^2\\
F_n(x)&=(x+1)^n-1+x^2.
\end{align*}
\end{proof}

Again, arranging the coefficients of $ F_n(x) $ as previously, we have \hyperref[tab:gpfoil]{Table \ref*{tab:gpfoil}}.
\begin{table}[H]
\centering
\[
\begin{array}{c|rrrrrrrrrrrrr}
n\ \backslash\ k	&0	&1	&2	&3	&4	&5	&6	&7	&8	&9	&10	&11	&12\\
\midrule
0	&0	&0	&1	&	&	&	&	&	&	&	&	&	&\\
1	&0	&1	&1	&	&	&	&	&	&	&	&	&	&\\
2	&0	&2	&2	&	&	&	&	&	&	&	&	&	&\\
3	&0	&3	&4	&1	&	&	&	&	&	&	&	&	&\\
4	&0	&4	&7	&4	&1	&	&	&	&	&	&	&	&\\
5	&0	&5	&11	&10	&5	&1	&	&	&	&	&	&	&\\
6	&0	&6	&16	&20	&15	&6	&1	&	&	&	&	&	&\\
7	&0	&7	&22	&35	&35	&21	&7	&1	&	&	&	&	&\\
8	&0	&8	&29	&56	&70	&56	&28	&8	&1	&	&	&	&\\
9	&0	&9	&37	&84	&126	&126	&84	&36	&9	&1	&	&	&\\
10	&0	&10	&46	&120	&210	&252	&210	&120	&45	&10	&1	&	&\\
11	&0	&11	&56	&165	&330	&462	&462	&330	&165	&55	&11	&1	&\\
12	&0	&12	&67	&220	&495	&792	&924	&792	&465	&220	&66	&12	&1
\end{array}
\]
\label{tab:gpfoil}
\caption{Values of $ f_{n,k} $ for $ 0\leq n\leq k\leq 12 $.}
\end{table}

First of all, notice from the formula \eqref{eq:foil} that the coefficients in \hyperref[tab:gpfoil]{Table \ref*{tab:gpfoil}} are those of the binomial expansion where some alterations appear at the monomials $ f_{0,n}$, $f_{0,1}x $ and $ f_{n,2}x^2 $. We write
\[
f_{n,k}=\begin{cases}
\binom{n}{1}-1,& \textit{if $ k=1 $};\\
\binom{n}{2}+1,& \textit{if $ k=2 $};\\
\binom{n}{k},& \textit{if $ k\geq 3 $}.
\end{cases}
\]

We have the following observations.

\begin{itemize}
\item The common constraint $ f_{n,0}=0 $, $ n\geq 0 $, which gives the sequence \seqnum{A000004} as previously:
\[0, 0, 0, 0, 0, 0, 0, 0, 0, 0, 0, 0, 0, 0, 0, 0, 0, 0, 0, 0, \ldots\]

\item The third column $ (f_{n,1})_{n\geq0}$ which is the sequence of nonnegative integers \seqnum{A001477}:
\[0,1,2,3,4,5,6,7,8,9,10,11,12,13,14,15,16,17,\ldots\]
We have $ f_{n,1}=n$ for any value of $ n $. When $ n\geq1 $, the sequence $ (f_{n,1})_{n\geq0} $ represents the number of $ 1 $-states that is obtained by applying a split of type $ A $ at a chosen crossing and by applying a split of type $ B $ at the remaining $ n-1 $ crossings. There are $ \binom{n}{1} $ possibilities of such operation. Since the foil $ F_0 $ is already a mutually disjoint union of two unknots, then we justify the value $ f_{0,1}=0 $.

\item The third column $ (f_{n,2})_{n\geq1} $ is known as the \textit{lazy caterer's sequence} \seqnum{A000124}:
\[1, 2, 4, 7, 11, 16, 22, 29, 37, 46, 56, 67, 79, 92, 106,\ldots\]
We will give more details about the splits sequence of the $ 2$-states in section \ref{sec:states}. We have $ f_{0,2}=1 $ because the knot $ F_2 $ is itself a $ 2 $-state.

\item The remaining columns $ (f_{n,k})_{n\geq k} $ where $ k\geq3 $ represent the usual binomial coefficients which express here the number of $ k $-states. We first choose $ k $ crossings at which we apply an $ A $-split. The result is a mutually disjoint of $ k $ twist loops. The $ k $-state is then obtained by applying a $ B $-split at each of these $ k $ twist loop crossing. For instance, consider the following illustration. 
	
\begin{figure}[H]
\centering
\includegraphics[width=0.9\linewidth]{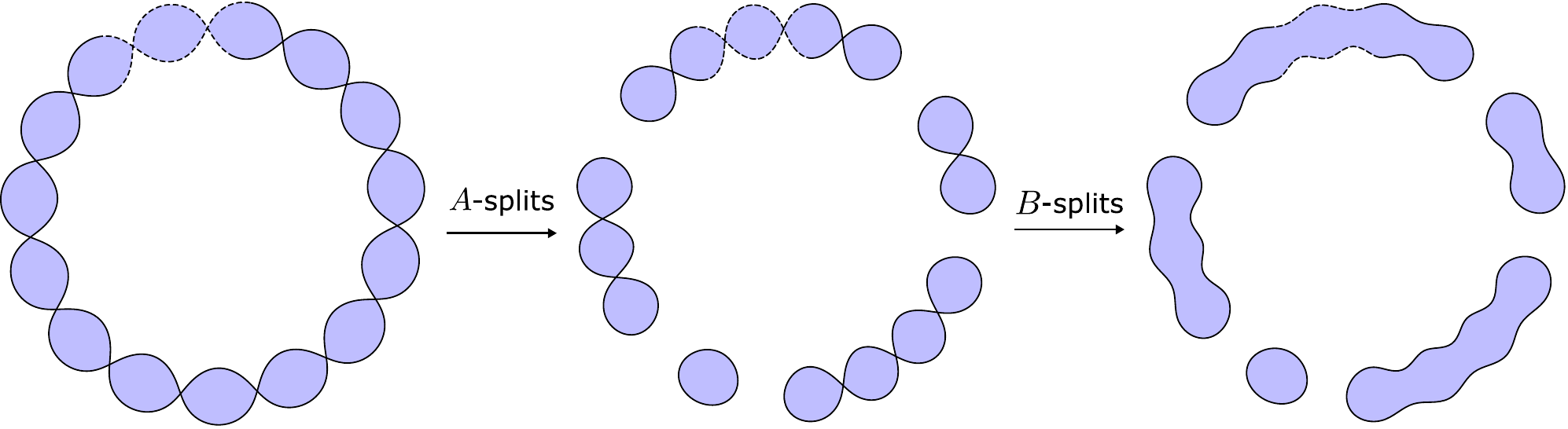}
\caption{A series of $ k $ $ A $-splits and $ n-k $ $ B $-splits which produces $ k $ components.}
\label{fig:Asplit}
\end{figure}

\item Let us also look at the diagonal $ (t_{n,n})_{n\geq 0} $: 
\[0,1,2,1,1,1,1,1,1,1,1,1,1,1,1,1,1,1,1,1,1,1,1,\ldots\]
\begin{itemize}
\item $ t_{0,0}=0 $ represents the usual constraint.

\item $ t_{1,1}=1 $ corresponds to the $ 1 $-state that is obtained by applying a $ B $-split at the single crossing. See \hyperref[sub:t11]{Figure \ref*{sub:t11}}.

\item $ t_{1,1}=2$ refers to the $ 2 $-states that are obtained by a couple of $ B $-splits or a couple of $A $-splits. See \hyperref[sub:t22]{Figure \ref*{sub:t22}}.

\item $ t_{n,n}=1,\ n\geq 3 $, correspond to the unique splits sequence that gives an $ n $-state. It consists of a ``full'' $ A $-split. See \hyperref[sub:Asplitfoil]{Figure \ref*{sub:Asplitfoil}}.

\begin{figure}[H]
\centering
\hspace*{\fill}
\subfigure[$ n=1 $\label{sub:t11}]{\includegraphics[width=0.11\linewidth]{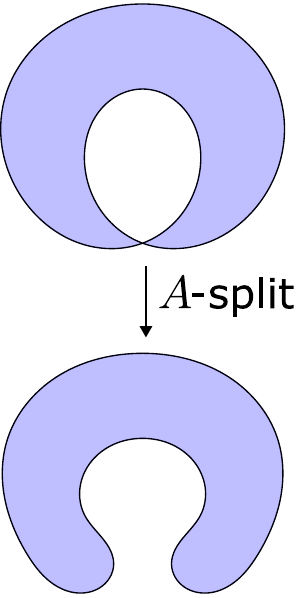}}\hfill%
\subfigure[$ n=2 $\label{sub:t22}]{\includegraphics[width=0.3\linewidth]{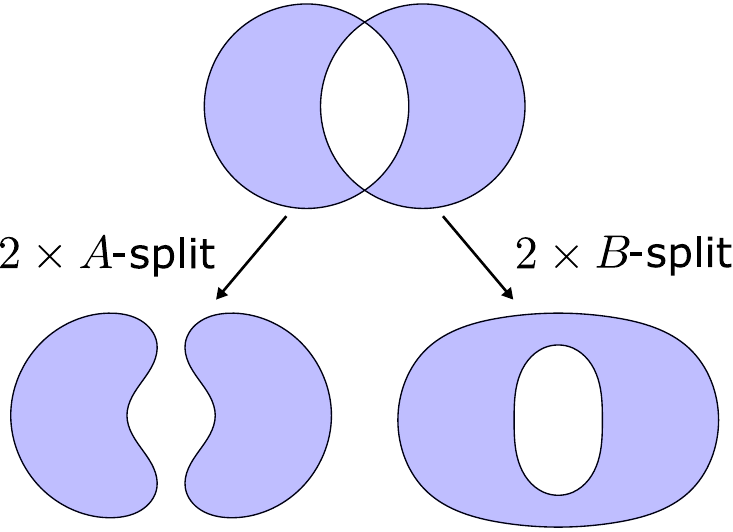}}\hfill%
\subfigure[$ n\geq3$\label{sub:Asplitfoil}]{\includegraphics[width=0.4\linewidth]{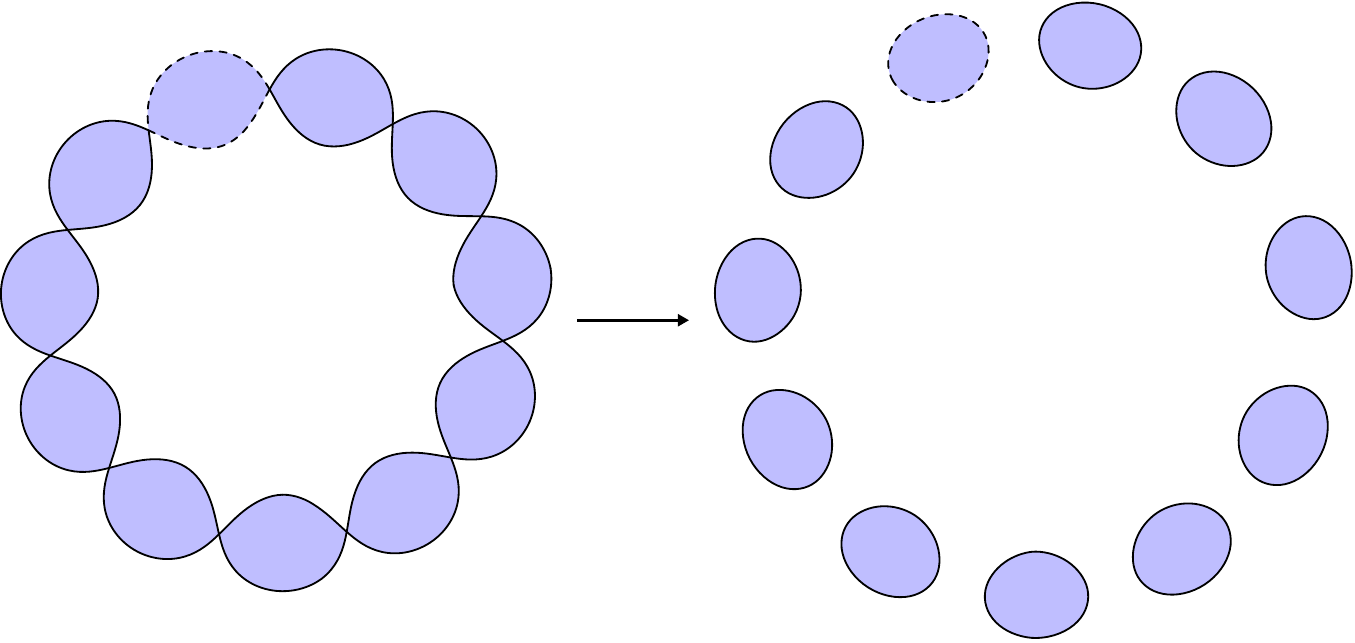}}
\hspace*{\fill}
\caption{The occurrences of the $ n $-states.}
\label{fig:t012}
\end{figure}

\end{itemize}
\end{itemize}

Let us now focus on the lazy caterer's sequence. Geometrically such sequence gives the number of regions on the plane defined by a set of lines in general arrangement. The next section covers some geometric properties of a family of lines in such arrangement.

\section{Lines in general arrangement}\label{sec:particularset}

Steiner \cite{steiner} proposed the following problem in 1826: what is the maximum number of pieces of a circle, a pancake or a pizza that
can be made with a given number of straight cuts? If we let $ P_n $ denote the maximum number of pieces when a circle is cut $ n $ times, then the answer is given by the recurrence relation
\begin{equation}\label{eq:line}
P_n=P_{n-1}+n,\mbox{ where } P_0=1. 
\end{equation}
The number $ P_n $ has the following closed form
\begin{equation}\label{eq:line-closed}
P_n=\dfrac{n^2+n+2}{2}
\end{equation}
which gives the lazy caterer's sequence \seqnum{A000124}:
\[1, 2, 4, 7, 11, 16, 22, 29, 37, 46, 56, 67, 79, 92, 106,\ldots\]

The relation \eqref{eq:line} describes the maximum number $ P_n $ of regions defined by $ n $ lines in the plane \cite[p.\ 4]{Concrete} in accordance to the following geometrical construction: we suppose there are already $n-1$ lines subject to a particular geometric configuration that guarantees the maximality of the planar regions number. When we add a new line, the latter increases the number of regions by $k$ if and only if it intersects each of the other lines in $ k-1 $ different places. It is clear that for $ k\leq n $, the number of regions will be the largest only if there is no point common to three lines. Therefore the new line may cut each of the $ n-1 $ lines in at most $ n-1 $ different points, i.e.,
\[
P_n\leq P_{n-1}+n.
\]
This upper bound is reached when the lines meet the following criteria.

\begin{definition}
We say that straight lines are in \textit{general arrangement} if no two lines are parallel, and no	three lines are concurrent.
\end{definition} 

Let us then construct a particular set of straight lines in general arrangement such that the resulting regions can easily be encoded.

\begin{definition}
The line $\mathcal{L}: ax+by+c=0 $ divides the Cartesian plane into two \textit{half-planes} characterized by the sign of $ H(x,y)=ax+by+c $. One of these half-planes satisfies $ ax+by+c>0 $ while the other half-plane satisfies $ ax+by+c<0 $. The line $ ax+by+c=0 $ is called the \textit{boundary} of the half-planes. 

The set of points in the Cartesian plane which lie above the line forms the \textit{upper half-plane}, and the set of points which lie below the line forms the \textit{lower half-plane}. 

A \textit{test point}, which lies on one side or the other of this line, is used to decide whether the solution of the inequality is the region above or below the straight line.
\end{definition}

\begin{notation}
Let $ (\mathcal{L}_n)_{n\geq0} $ denote the family of lines of equation $ y=l_n(x)=nx-n^2 $, and let $ (\mathcal{H}_n)_{n\geq0} $ denote the family of half-planes defined by $H_n(x,y)=y-nx+n^2 $.
\end{notation}

The following proposition exhibits an immediate property of this family of straight lines.

\begin{proposition}
The lines $(\mathcal{L}_n)_{n\geq0} $ are in general arrangement.
\end{proposition}

\begin{proof} 
It is clear that for a couple of nonnegative integers $ (p,q) $ such that $ p\neq q $, the corresponding lines $ \mathcal{L}_p $ and $ \mathcal{L}_q $ have different slope and meet at the point $ \big(p+q,pq\big) $. 
	
Now, let $ r $ be another nonnegative integer such that $ r\neq p $ and $ r\neq q $. If the line $ \mathcal{L}_r $ passes through the point $ \big(p+q,pq\big), $ then the following equality must hold
\begin{equation}\label{eq:2ndeg}
r^2-(p+q)r+pq=0.
\end{equation} However, the roots of the quadratic equation \eqref{eq:2ndeg} are $ r=p $ and $ r=q $. Therefore, triple points do not occur.
\end{proof}

\begin{notation}
Let $\tau (-1,1) $ be the test point. From the inequality $$ H_n(-1,1)=1+n+n^2>0 \mbox{ where } n\geq0,$$ we let $ \mathcal{H}_n^+ $ denote the upper half-plane that contains $ \tau $, and let $ \mathcal{H}_n^- $ denote the lower half-plane, both characterized by $ H_n(x,y) $.
\end{notation}

\begin{proposition}\label{prop:upperintersection}
For three nonnegative integers $ p$, $ q $, $r $ such that $ p<q<r $, the intersection point $ (p+q,pq) $ of the lines $ \mathcal{L}_p $ and $ \mathcal{L}_q $ is located at the upper half-plane $ \mathcal{H}_r^+ $.
\end{proposition}

\begin{proof}
Without loss of generality, we may fix $ r $ and $ q $ then show that for any $ \widehat{p}<q $ the inequality $ H_r(\widehat{p},q)>0 $ holds. We have 
\begin{equation}\label{eq:affine}
H_r(\widehat{p}+q,\widehat{p}q)=\widehat{p}(q-r)+r(r-q).
\end{equation}
The sign of the affine function \eqref{eq:affine} with respect to the variable $ \widehat{p} $ is given by the sign of $ q-r $ for $ \widehat{p}<r $. Hence
\[
H_r(\widehat{p}+q,pq)>0\ \mbox{ for }\widehat{p}<r.
\]
\end{proof}

\begin{proposition}\label{prop:arrangement}
Let $ n $ lines be arranged on the plane, and let $ (a,b)$ be a point inside the lower half-plane $ \mathcal{H}_n^- $ such that $ y-kx+k^2\neq0$ for any nonnegative integer $ k $ verifying $ k\leq n $ (this is to ensure that the point $ (a,b) $ does not lie on a boundary line). 

There exists a nonnegative integer $ r $ satisfying one of the following conditions:
\begin{enumerate}
\item either $ (a,b)\in \mathcal{H}_{r}^+\cap \mathcal{H}_{r+1}^- $,

\item or $ (a,b)\in \mathcal{H}_0^- $.
\end{enumerate}
\end{proposition}

\begin{proof}
Since all the intersection points are located at the upper plane $ \mathcal{H}_n^+ $, the key ingredient is to show that for any nonnegative integers $ p$, $q $ such that $ p<q<n $, and for a any point $(x,y) $ located at the lower half-plane $ \mathcal{H}_n^-$, the inequality $l_p(x)<l_q(x) $ holds when $ x>p+q $. 

We verify such inequality by writing
\begin{equation}\label{eq:linear}
l_p(x)-l_q(x)=(p-q)x-\left(p^2-q^2\right).
\end{equation}
The root of the linear polynomial \eqref{eq:linear} is $ x=p+q $. Therefore, we have $ l_p(x)-l_q(x)>0 $ when $ x>p+q $. This means that $ \mathcal{H}_n ^-$ is partitioned in way that the half-lines which start from their respective intersection point with the line $ \mathcal{L}_n $ are arranged one above another. We conclude by distinguishing the sign of $ b $.

\begin{enumerate}
\item If $b>0 $ and $ (a,b) \in \mathcal{H}_n^-$, then there exists a nonnegative integer $ r $ such that we can recursively write the following inequalities
\[
l_0(a)<l_1(a)<l_2(a)<\cdots<l_r(a)<b< l_{r+1}(a).
\]
The last two inequalities imply that $(a,b)\in \mathcal{H}_{r}^+\cap \mathcal{H}_{r+1}^-$.

\item If $ b<0 $ and $ (a,b) \in \mathcal{H}_n^-$, then $ (a,b)\in \mathcal{H}_n^-\cap \mathcal{H}_0^- $.
\end{enumerate}
\end{proof}

\section{Encoding the regions on the plane}\label{sec:planeencoding}

Let us now encode each region of the plane when a finite lines of the family $ (\mathcal{L}_n)_{n\geq0} $ are drawn. We associate a region with a binary word of length $ n $ whose letters are from the set $ \{0,1\} $. The $ i $-th letter $ \sigma_i $, where $ 0\leq i\leq n-1, $ indicates whether the region lies or not at either of the $i $-th lower or upper half-plane:
\begin{itemize}
\item if the region lies at the lower half-plane $ \mathcal{H}_{i}^- $, then we write $ \sigma_i =0$;

\item if the region lies at the upper half-plane $ \mathcal{H}_{i}^+ $, then we write $ \sigma_i =1$.
\end{itemize}

We let $ \mathcal{P}_n $ denote the set of such words. This set labels the resulting regions when the lines $ \mathcal{L}_0,\mathcal{L}_1,\ldots,\mathcal{L}_{n-1}$ are arranged on the plane. A region is therefore uniquely identified with a words from the set $ \mathcal{P}_n $. 

As an example, let us define the set $ \mathcal{P}_4 $ with the help of \hyperref[fig:P3]{Figure \ref*{fig:P3}}. We use $ \tau(-1,1) $ as the test point.

\begin{figure}[H]
\centering
\includegraphics[width=.9\linewidth]{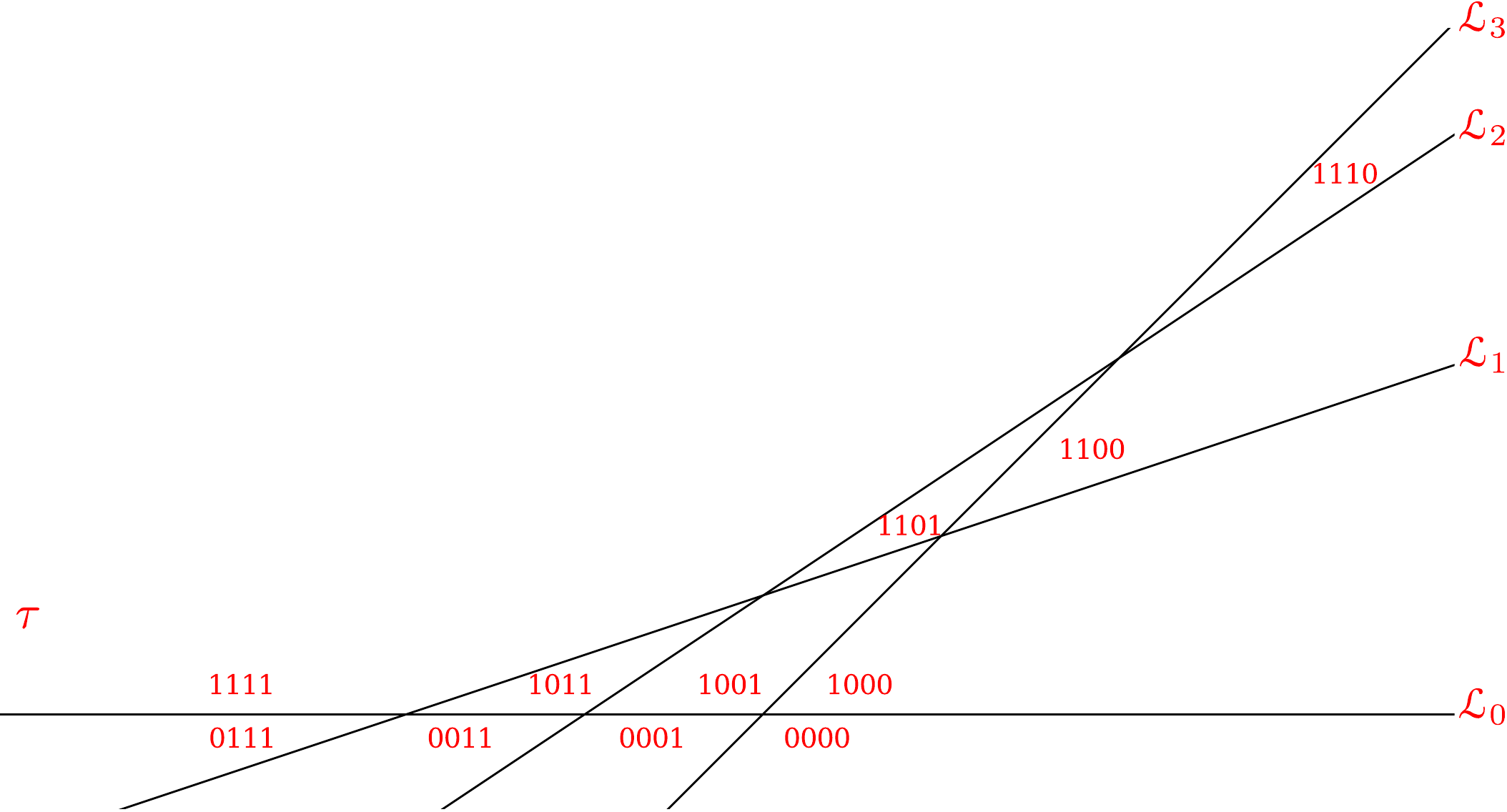}
\caption{The encodings of the $ 11 $ planar regions.}
\label{fig:P3}
\end{figure}

We see that
\[
\mathcal{P}_4=\{0000,1000,1100,1110,0001,1001,1101,0011,1011,0111,1111\}.
\]

\begin{remark}
The set $ \mathcal{P}_0 $ consists of the regions on the plane when there are no line, i.e., the only region is the plane itself. As a convention, we write $ \mathcal{P}_0=\{\varepsilon\}$ where $ \varepsilon $ denotes the empty word.
\end{remark}

Furthermore, we equip the set $ \mathcal{P}_n $ with the usual words concatenation. The encoding arguments in this paper make extensive use of the following notation. 

\begin{notation}
Let $ \sigma \in\{0,1\} $, and let $ k$, $\ell $ be integers. Also, we let $ w $ denote a binary words defined other the alphabet $ \{0,1\} $. The concatenating operation verifies
\begin{itemize}
\item $\sigma^k= \underbrace{\sigma\sigma\cdots\sigma\sigma\sigma}_{\text{$ k $ times}} $;

\item $ \sigma^k\sigma^\ell=\sigma^{k+\ell} $;

\item $ \sigma^k=\varepsilon\mbox{ if }k\leq0$;

\item $ \varepsilon^k=\varepsilon $;

\item $ w\varepsilon=\varepsilon w =w$.
\end{itemize}
\end{notation}

We now suppose that each region is already associated with the previously defined binary words. We have the following proposition.

\begin{proposition}
Let $ n-1 $ lines be arranged on the plane. By adding a $ n $-th line, we add $ n $ new regions that are encoded by 
\[
\pi_k=1^k0^{n-k},\ k=0,1,\ldots, n-1. 
\]
The remaining regions are simply the previous ones whose codes are rewritten by appending the $ 1 $-digit at the right end.
\end{proposition}

\begin{proof}
The new regions encoding is an application of \hyperref[prop:arrangement]{Proposition \ref*{prop:arrangement}}. The lower half-plane $ \mathcal{H}_{n}^- $ is divided in a way that the lines are arranged in ascending order. More precisely, these regions are
\[
\bigcap\limits_{i=0}^{n-1}\mathcal{H}_{i}^- \ \mbox{and}\ \left(\bigcap\limits_{i=0}^{j}\mathcal{H}_{i}^+\right)\cap\left(\bigcap\limits_{i=j+1}^{n-1}\mathcal{H}_{i}^- \right),\ j=0,1,\ldots,n-1.
\]
Therefore, we have the correspondence
	
\begin{align*}
\bigcap\limits_{i=0}^{n-1}\mathcal{H}_{i}^- \Leftrightarrow& 0^{n}\\
\mathcal{H}_{0}^+\cap\left(\bigcap\limits_{i=1}^{n-1}\mathcal{H}_{i}^- \right)\Leftrightarrow &10^{n-1}\\
\mathcal{H}_{0}^+\cap\mathcal{H}_{1}^+\cap\left(\bigcap\limits_{i=2}^{n-1}\mathcal{H}_{i}^- \right)\Leftrightarrow &1^20^{n-2}\\
\vdots&\\
\mathcal{H}_{0}^+\cap\mathcal{H}_{1}^+\cap\cdots\cap\mathcal{H}_{n-3}^+\cap\mathcal{H}_{n-2}^-\cap\mathcal{H}_{n-1}^-\Leftrightarrow& 1^{n-2}0^{2}\\
\mathcal{H}_{0}^+\cap\mathcal{H}_{1}^+\cap\cdots\cap\mathcal{H}_{n-3}^+\cap\mathcal{H}_{n-2}^+\cap\mathcal{H}_{n-1}^-\Leftrightarrow& 1^{n-1}0.
\end{align*}
Now according to \hyperref[prop:upperintersection]{Proposition \ref*{prop:upperintersection}}, all the old regions are located at the upper half-plane $ \mathcal{H}_{n-1}^+ $. Hence, the corresponding encoding are written as $ \pi1 $ for any $ \pi\in\mathcal{P}_{n-1} $.
\end{proof}

\begin{remark}
We let $ \mathcal{R}_n $ denote the set $ \left\{1^k0^{n-k}\mid 0\leq k\leq n-1 \right\} $ which represents the $ n $ new regions. 
We may write
\begin{equation}\label{eq:recP}
\mathcal{P}_0=\{\varepsilon\},\ \mathcal{P}_n=\mathcal{P}_{n-1}1\cup\mathcal{R}_n\ \mbox{where}\ n\geq1.
\end{equation}

Since $ \# \mathcal{P}_0=1$, then we recognize the lazy caterer's sequence recurrence
\begin{align*}
\#\mathcal{P}_n&=\#\mathcal{P}_{n-1}+n\\
&=\dfrac{n(n+1)+2}{2}\cdot
\end{align*}

Notice that the relation \eqref{eq:recP} allows us to recursively compute $ \mathcal{P}_{n-1} $ from $ \mathcal{P}_{n} $. It suffices to extract the words of $ \mathcal{P}_{n} $ which end with the $ 1 $-digit and remove this digit so that the length of the word become $ n-1 $. For example, using the result in \hyperref[fig:P3]{Figure \ref*{fig:P3}}, we have
\begin{align*}
\mathcal{P}_4&=\{0000,1000,1100,1110,0001,1001,1101,0011,1011,0111,1111\};\\
\mathcal{P}_3&=\{000,100,110,001,101,011,111\};\\
\mathcal{P}_2&=\{00,10,01,11\};\\
\mathcal{P}_1&=\{0,1\};\\
\mathcal{P}_0&=\{\varepsilon\},\ \mbox{we use the property}\ 1=1\varepsilon.
\end{align*}
\end{remark}

We can now write the encodings of the planar regions defined by $ n $ straight lines.

\begin{corollary}
Let $ n\geq 1 $. The encodings of the plane defined by $ n $ lines is given by
\begin{equation}\label{eq:plancode}
\mathcal{P}_n:=\left\{1^k0^{n-p-k}1^p\mid 0\leq p\leq n-1\ and\ 0\leq k\leq n-p-1\right\}\cup\left\{1^n\right\}.
\end{equation}
\end{corollary}

\begin{proof}
Let us prove by induction that for all positive integer $ n $, the identity (\ref{eq:plancode}) holds. 

When $ n=1 $, we have $ 0\leq p\leq 0 $ and $ 0\leq k\leq 0 $ so that the only candidate is the word $ 0 $. So we write
\[
\mathcal{P}_1=\{0,1\}
\]
which coincides with the disjoint half-planes characterized by the line $ \mathcal{L}_0:y=0 $.

Assume that the formula \eqref{eq:plancode} is verified for $ n=\ell $, and let us show that it still holds for $n=\ell+1 $.
From the induction hypothesis we have $$ \mathcal{P}_\ell=\left\{1^{k}0^{\ell-p-k}1^p\mid0\leq p\leq \ell-1\ and\ 0\leq k\leq \ell-p-1\right\}\cup\left\{1^\ell\right\}.$$By adding a $( \ell+1) $-th lines, we obtain $ \ell+1 $ new regions, namely $$ \mathcal{R}_{\ell+1}=\left\{1^k0^{\ell+1-k}\mid{0\leq k\leq \ell}\right\}, $$ and we rewrite the old regions as
\begin{align*}
\mathcal{P}_\ell1&=\left\{1^{k}0^{\ell-p-k}1^p1\mid0\leq p\leq \ell-1\ and\ 0\leq k\leq \ell-p-1\right\}\cup1\left\{1^\ell\right\}\\
&=\left\{1^{{k}}0^{\ell-\widehat{p}+1-k}1^{\widehat{p}}\mid 1\leq \widehat{p}\leq \ell\ and\ 0\leq {k}\leq \ell-\widehat{p}\right\}\cup\left\{1^{\ell+1}\right\}.
\end{align*}
Therefore 
\begin{align*}
\mathcal{P}_{\ell+1}&=\mathcal{P}_{\ell}1\cup\left\{1^k0^{\ell+1-k}\mid{0\leq k\leq \ell}\right\}\\
&=\left\{1^{{k}}0^{\ell+1-\widehat{p}-k}1^{\widehat{p}}\mid 0\leq \widehat{p}\leq \ell\ and\ 0\leq {k}\leq \ell-\widehat{p}\right\}\cup\left\{1^{\ell+1}\right\}.
\end{align*}
\end{proof}

Notice that we can actually rewrite the formula \eqref{eq:recP} as belows in order to prove the induction:
\begin{equation}\label{eq:generatorP}
\mathcal{P}_n=\left\{1^{n}\right\}\cup\left(\bigcup\limits_{p=0}^{n-1}\mathcal{R}_{n-p}1^{p}\right),\ n\geq 1.
\end{equation}

\begin{remark}
For a finite arbitrary family of lines in general arrangement, say of cardinality $ n $, we can always label each region in a unique way with words from $ \mathcal{P}_n $. Therefore, we shall abusively refer to the set $ \mathcal{P}_n $ as a partition of the plane.
\end{remark}

\section{The states of the foil family}\label{sec:states}

Similarly, we may write a sequence of binary digits that relates the sequence of splits we applied to a diagram. Here the binary digit is referring to the type of split that is applied at each crossing. We first assign an orientation and label the crossings from $ 1 $ to $ n $. If we apply an $ A $-split at the crossing number $ i $, then we write the $ 0 $-digit. If it is a $ B $-split, then we write the $ 1 $-digit. 

Let $ \mathsf{F}_n$ and $ \mathsf{T}_n$ respectively denote the set words associated with the states of the $ n $-foil and the $ n $-twist loop. From the relation \eqref{eq:foiltwist}, we have
\begin{equation}\label{eq:statesFT}
\mathsf{F}_n=0\mathsf{T}_{n-1}\cup1\mathsf{F}_{n-1}.
\end{equation}

As an illustration, let us write the set of words associated with the states of the \textit{quatrefoil} $ F_4 $ as pictured in \hyperref[fig:states4foil]{Figure \ref*{fig:states4foil}}. We apply the previously mentioned rule by splitting the crossing number $ 1 $, and then computing the states of the resulting planar isotopic $ 3 $-twist loop and $ 3 $-foil. The set $ \mathsf{F}_4$ is therefore given by
\begin{align*}
\mathsf{F}_4=\{&0000,0100,0001,0101,0010,0110,0011,0111,\\
&1000,1001,1010,1011,1100,1101,1110,1111\}.
\end{align*}

\begin{figure}
\centering
\includegraphics[width=0.8\linewidth]{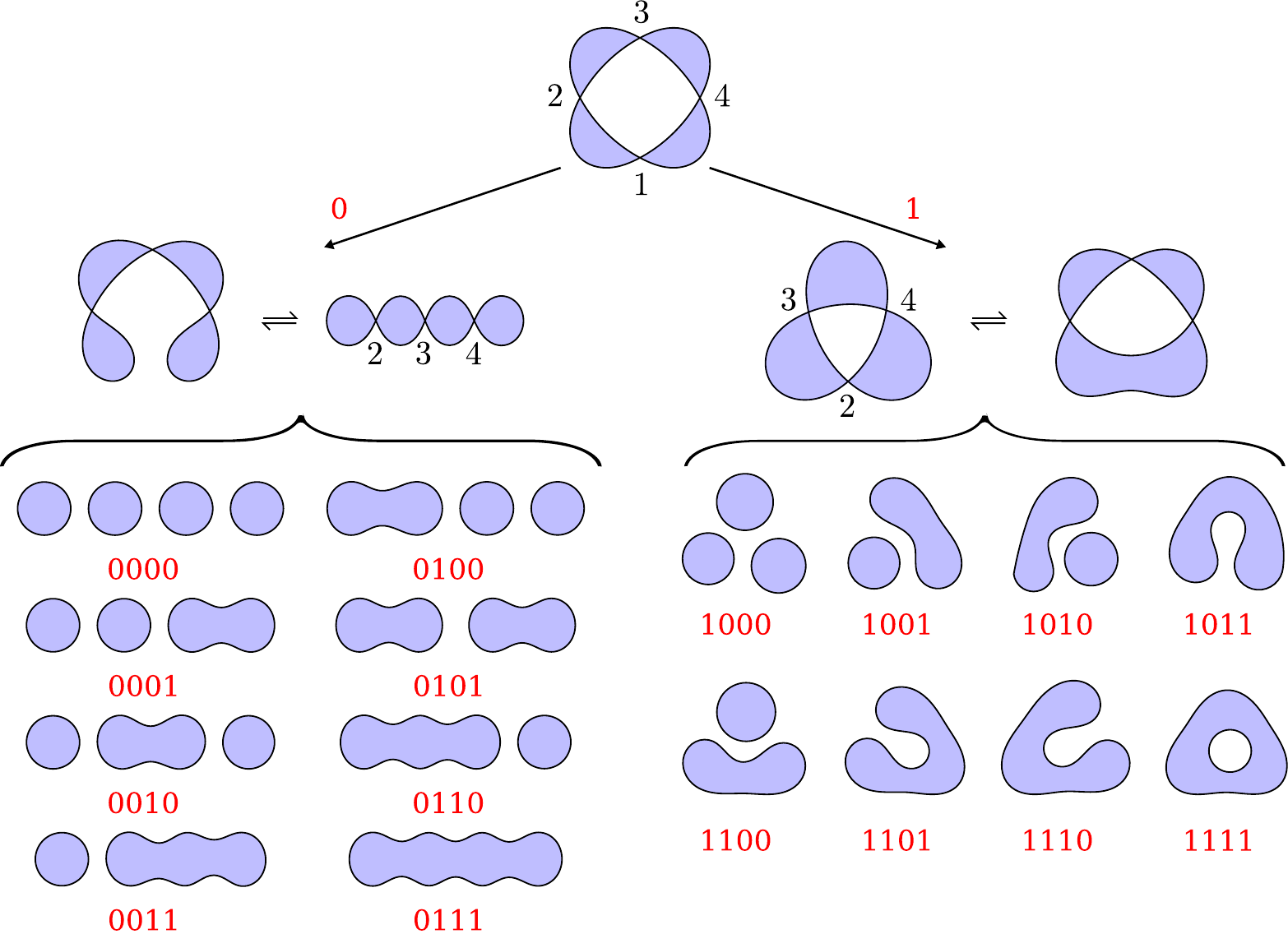}
\caption{The states of the quatrefoil.}
\label{fig:states4foil}
\end{figure}

We now let $ \mathcal{F}_n $ denote a subset of $ \mathsf{F}_n $ such that the elements of $ \mathcal{F}_n $ describe the sequences of splits that lead to a $ 2$-state. For the case of the previous quatrefoil, we have 
\begin{equation*}
\mathcal{F}_4=\left\{0101,0110,0011,1001,1010,1100,1111\right\}.
\end{equation*}
Generally, to end up with two Jordan curves, there exists two ways:
\begin{itemize}
\item either we choose two crossings at which we apply an $ A $-split, and then apply a $ B $-split at every remaining crossing;

\item or we simply apply a $ B$-split at every crossing.
\end{itemize}
The number of possibilities is therefore $\binom{n}{2} $ for the former and $ 1 $ for the later. 

From this point on, we may identify a state by its sequence of splits. Let us first establish the set of words for the $ n $-twist loop.

\begin{proposition}
The set of $ 2 $-states of the $ n $-twist loop is
\begin{equation*}
\mathcal{T}_n:=\left\{1^k01^{n-k-1}\mid 0\leq k\leq n-1\right\},\ n\geq2.
\end{equation*}
\end{proposition}

\begin{proof}
Firstly, recall that the generating polynomial of the $ n$-twist loop is expressed by $ T_n(x)=x(x+1)^n $. The coefficient of the term in $ x^2 $ is given by $\binom{n}{1} $. This means that in order to obtain two components, we have to choose one crossing and apply an $ A$-split. Then we apply a $ B $-split at each remaining crossing. Thus, the corresponding words are of the form $$ \omega_k=1^k01^{n-k-1},\ k=0,\ldots, n-1.$$
\end{proof}

\begin{remark}
As a convention, we write $ \mathcal{T}_0=\varnothing $ which we explain by the fact that not only the knot $ T_0 $ has no crossings, but also that the state of the knot $ T_0 $, actually itself, is excluded from the class of the $ 2 $-state. Besides, we have $\mathcal{T}_1=\{1\} $ since it requires one, and only one, split of type $ A $ to obtain a $ 2 $-state from the $ 1 $-twist loop.
\end{remark}

\begin{corollary}
For $ n\geq2 $, the set of the $ 2 $-states of the $ n $-foil is given by
\begin{equation}\label{eq:foilset}
\mathcal{F}_{n}:=\left\{1^p01^k01^{n-p-k-2}\mid 0\leq p\leq n-2\ and\ 0\leq k\leq n-p-2\right\}\cup \left\{1^{n}\right\}.
\end{equation}
\end{corollary}

\begin{proof}
We use induction to prove the identity \eqref{eq:foilset}. When $ n=2 $, the boundaries $ 0\leq p\leq 2-2 $ and $ 0\leq k\leq 2-p-2 $ imply 
\[
\mathcal{F}_2=\{00,11\}.
\]
From \hyperref[sub:t22]{Figure \ref*{sub:t22}}, we can see that a $2 $-foil ends up to $ 2 $-states if and only if we apply two successive $A$-splits or two successive $ B $-splits.

Let $ \ell\in\mathbb{N} $ be given, and assume that the identity \eqref{eq:foilset} is true for $ n = \ell $. Using the relation \eqref{eq:foiltwist} and \eqref{eq:statesFT}, we write
\begin{equation}\label{eq:foiltwistset}
\mathcal{F}_{\ell+1}=1\mathcal{F}_{\ell}\cup0\mathcal{T}_\ell.
\end{equation}
By the induction hypothesis, we have
\begin{align}
0\mathcal{T}_\ell:=\left\{01^k01^{\ell-k-1}\mid0\leq k\leq \ell-1\right\}\label{eq:0Cl}
\end{align}
and 
\begin{align}
1\mathcal{F}_{\ell}&:=\left\{1^{p+1}01^k01^{\ell-p-k-2}\mid0\leq p\leq \ell-2\ and\ \leq k\leq\ell-p-2\right\}\cup \left\{11^{\ell}\right\}\nonumber\\
&=\left\{1^{\widehat{p}}01^k01^{\ell-\widehat{p}-k-1}\mid 1\leq \widehat{p}\leq \ell-1\ and\ 0\leq k\leq\ell-\widehat{p}-1\right\}\cup \left\{1^{\ell+1}\right\}\label{eq:1Tl}.
\end{align}
Combining \eqref{eq:0Cl} and \eqref{eq:1Tl} we get
\begin{align*}
\mathcal{F}_{\ell+1}&=\left\{1^{\widehat{p}}01^k01^{\ell-\widehat{p}-k-1}\mid0\leq \widehat{p}\leq \ell-1\ and\ 0\leq k\leq\ell-\widehat{p}-1\right\}\cup \left\{1^{\ell+1}\right\}.
\end{align*}
\end{proof}

\begin{remark} 
It is worth mentioning following observations.

\begin{enumerate}
\item The $ 0 $-foil is already made out of two components and we write $\mathcal{F}_0= \{\varepsilon\}$.

\item For the case $ n=1 $, recall that the only split that produces two components is the one that opens the $ B $ channel. Therefore, we write $\mathcal{F}_1= \{1\}$.

\item From the previous observation, the set $ \mathcal{F}_{n} $ can be defined by the following recurrence:	
\begin{equation}\label{eq:recfoil}
\mathcal{F}_0=\left\{\varepsilon\right\},\ \mathcal{F}_1=\left\{1\right\},\ \mathcal{F}_{n}=0\mathcal{T}_{n-1}\cup1\mathcal{F}_{n-1}\mbox{ for }n\geq2.
\end{equation} 
It allows us to write the $ 2 $-states set of a $ (n-1) $-foil when the set $ \mathcal{F}_{n} $ is already given. A $ n-1 $ length words $ \omega $ is an element of $ \mathcal{F}_{n-1} $ if and only if $ 1\omega $ is an element of $ \mathcal{F}_{n} $. For example, we can recursively deduce the set $ \mathcal{F}_{1} $, $ \mathcal{F}_{2} $ and $ \mathcal{F}_{3} $ from $ \mathcal{F}_{4} $. We have
\begin{align*}
\mathcal{F}_4&=\left\{0101,0110,0011,1001,1010,1100,1111\right\};\\
\mathcal{F}_{3} &=\{001,010,100,111\};\\
\mathcal{F}_{2} &=\{00,11\};\\
\mathcal{F}_{1} &=\{1\}.
\end{align*}
Using the property $ 1\varepsilon=1 $ we justify the convention $ \mathcal{F}_{0}=\{\varepsilon\} $.

\item We can then calculate the cardinal of the set $ \mathcal{F}_{n} $ as belows:
\begin{align*}
\#\mathcal{F}_{n}&=\#0\mathcal{T}_{n-1}+\#1\mathcal{F}_{n-1}\\
&=\#1\mathcal{F}_{n-1}+n-1\\
&=\dfrac{n(n-1)+2}{2}
\end{align*}
which gives the lazy caterer's sequence when $n\geq 1 $.

\item Let $ n\geq 1 $. We can combine \eqref{eq:foilset} and \eqref{eq:recfoil} as belows:
\begin{equation}\label{eq:generatorF}
\mathcal{F}_{n+1}=\left\{1^{n+1}\right\}\cup\left(\bigcup\limits_{p=0}^{n-1}1^{p}0\mathcal{T}_{n-p}\right).
\end{equation}
\end{enumerate}	
\end{remark}

As we see, not only we find the lazy caterer's sequence but also the encoding of the $ 2 $-states. The latter, apparently share similarities with the encoding of the planar regions defined by lines in general arrangement. The following section aims at establishing a connexion between them.

\section{Constructing the bijection}\label{sec:bijection}
To begin with, let us focus on the following recurrence definition:
\begin{equation*}
\mathcal{P}_0=\left\{\varepsilon\right\},\ \mathcal{P}_n=\left\{1^{n}\right\}\cup\left(\bigcup\limits_{p=0}^{n-1}\mathcal{R}_{n-p}1^{p}\right)\mbox{ for }n\geq1,
\end{equation*}
and
\begin{equation*}
\mathcal{F}_1=\left\{1\right\},\ \mathcal{F}_{n+1}=\left\{1^{n+1}\right\}\cup\left(\bigcup\limits_{p=0}^{n-1}1^{p}0\mathcal{T}_{n-p}\right)\mbox{ for }n\geq1.
\end{equation*}

Intuitively, both of these identities suggest that an initial step for constructing a bijection is to isolate a sub-word from $ \mathcal{R}_{n-\widetilde{p}} $ and $\mathcal{T}_{n-\widetilde{p}}$ for a certain nonnegative integer $ \widetilde{p} $ verifying $ 0\leq \widetilde{p}\leq n-1 $. 

If $ n=0 $, then the map $\phi:\mathcal{P}_0\longrightarrow\mathcal{F}_1$ defined by $ \phi(\varepsilon)=1 $ is obviously a bijection. 

If $ n\geq 1 $, then we need the following construction.

\begin{lemma}\label{lemma:bij}	
Let $ \ell $, $ r $, $ s $ be three nonnegative integers such that $ 0\leq r\leq \ell -1$ and $ 0\leq s\leq \ell -1$, and let
\begin{itemize}
\item 	$ \pi_{\ell,r} =1^r0^{\ell-r}\in\mathcal{R}_\ell$,

\item $ \omega_{\ell,s} =1^s01^{\ell-s-1}\in\mathcal{T}_{\ell}$
\end{itemize}
respectively denote the element of $\mathcal{R}_\ell$ and $\mathcal{T}_{\ell}$.

We define the following map:
\begin{align*}
\phi_\ell:\left\{1^k0^{\ell-k}\ | \ 0\leq k\leq \ell-1\right\}&\longrightarrow \left\{1^k01^{\ell-k-1}\mid 0\leq k\leq \ell-1\right\}\\
\pi_{\ell,k}&\longmapsto\omega_{\ell,k}.
\end{align*}
Then, the map $ \phi_\ell$ is a bijection.
\end{lemma}

\begin{proof}
Both words $ \pi_{\ell,r} $ and $ \omega_{\ell,s}$ are of length $ \ell $ and the correspondence is one-to-one as we browse the index $ k $. Since $ 0\leq k\leq \ell-1 $ we construct the correspondence as follows:
\[
1^k0^{\ell-k}\longrightarrow1^k00^{\ell-k-1}\longrightarrow1^k01^{\ell-k-1}
\]
If $ k=\ell-1 $, then we map $ 1^{\ell-1}0 $ with $ 1^{\ell-1}00^0=1^{\ell-1}0 $.
\end{proof}

Let us now define a map from $ \mathcal{P}_n $ to $ \mathcal{F}_{n+1}$.

\begin{definition}
When $ n\geq 1 $, we let the map $ \phi $ be
\begin{align*}
\phi:\mathcal{P}_n&\longrightarrow \mathcal{F}_{n+1}\\
\pi&\longmapsto \omega= \begin{cases}
1^{n+1}, & \text{if $\pi=1^n$;}\\
1^p01^{k}01^{n-p-k-1}, & \text{if $ \pi= 1^{k}0^{n-p-k}1^p$.}
\end{cases}
\end{align*}
We associate a planar region with a $ 2 $-state according to the transformation below:
\begin{enumerate}
\item remove the last $ p $ $1$-digits;
\item change the last $ n-p-k-1 $ $ 0$-digits to $1 $;
\item concatenate the words $1^p0$ at the left-end.
\end{enumerate}
\end{definition}

\begin{proposition}
The map $ \phi $ is a bijection from $ \mathcal{P}_n $ to $ \mathcal{F}_{n+1}$.
\end{proposition}

\begin{proof}
Recall that for $ n\geq 1 $, we have
\begin{equation*}
\mathcal{P}_n:=\left\{1^k0^{n-p-k}1^p\mid\ 0\leq p\leq n-1\ and\ 0\leq k\leq n-p-1\right\}\cup\left\{1^n\right\},
\end{equation*}
and
\begin{equation*}
\mathcal{F}_{n+1}:=\left\{1^p01^k01^{n-p-k-1}\mid 0\leq p\leq n-1\ and \ 0\leq k\leq n-p-1\right\}\cup \left\{1^{n+1}\right\}.
\end{equation*}
Let $ 0\leq \widetilde{p}\leq n-1 $ and consider the map
\begin{align*}
\phi_{n-\widetilde{p}}:\left\{1^k0^{n-\widetilde{p}-k}\ | \ 0\leq k\leq n-\widetilde{p}-1\right\}&\longrightarrow \left\{1^k01^{n-\widetilde{p}-k-1}\mid 0\leq k\leq n-\widetilde{p}-1\right\}\\
\pi_{n-\widetilde{p},k}&\longmapsto\omega_{n-\widetilde{p},k}
\end{align*}
which, according to \hyperref[lemma:bij]{Lemma \ref*{lemma:bij}}, is a bijection.

The map $ \phi $ then can be defined as follows:
\begin{itemize}\label{key}
\item $ \phi(1^n)=1^{n+1} $;
\item $ \phi(\pi)=1^p0\phi_{n-p}(\pi_{n-p,k})$ where $ \pi= \pi_{n-p,k}1^p$.
\end{itemize}
\end{proof}

Now, let $ \omega $ be a word of the set $ \mathcal{F}_{n+1} $, and let $ \pi $ be the inverse image of $ \omega $ by the map $ \phi $. For two nonnegative integers $ n$ and $k $, the construction of the map $ \phi^{-1} $ is straightforward:
\begin{enumerate}
\item delete the first $ p $ $ 1 $-digits as well as the following $ 0 $-digit;
\item change the last $n-p-k-1 $ $ 1 $-digits to $ 0 $;
\item append the word $ 1^p$ at the right-end.
\end{enumerate}

Thus we write
\begin{align*}
\pi=\phi^{-1}(\omega)= \begin{cases}
1^{n}, & \text{if $\omega=1^{n+1}$;}\\
1^{k}0^{n-p-k}1^p, & \text{if $ \omega= 1^p01^{k}01^{n-p-k-1}$.}
\end{cases}
\end{align*}

\begin{example}Let us take $ n=7 $. 
	
We first write the set $\bigcup\limits_{p=0}^{6}\mathcal{R}_{7-p}1^{p} $ in \hyperref[tab:P7]{Table \ref*{tab:P7}} as well as the set $\bigcup\limits_{p=0}^{6}1^{p}0\mathcal{T}_{7-p}$ in \hyperref[tab:F8]{Table \ref*{tab:F8}}. The bijection is then performed entrywise. For instance, we map the entry $ 1100001 $ with $ 10110111 $, here $ p=1 $ and $ k=2 $.
	
\begin{table}[H]
\centering
{
\def\arraystretch{1.25}
$ \begin{array}{c|ccccccc}
k\ \backslash\ p & 0 & 1 & 2 & 3 & 4 & 5 & 6\\ \midrule
0 & 0000000 & 0000001 & 0000011 & 0000111 & 0001111 & 0011111 & 0111111\\
1 & 1000000 & 1000001 & 1000011 & 1000111 & 1001111 & 1011111 & \\
2 & 1100000 & \hyperref[entry:F12]{\textcolor{red}{1100001}}\label{entry:P12} & 1100011 & 1100111 & 1101111 & & \\
3 & 1110000 & 1110001 & 1110011 & 1110111 & & & \\
4 & 1111000 & 1111001 & 1111011 & & & & \\
5 & 1111100 & 1111101 & & & & & \\
6 & 1111110 & & & & & &
\end{array} $
}
\caption{The set $\bigcup\limits_{p=0}^{6}\mathcal{R}_{7-p}1^{p} $.}
\label{tab:P7}
\end{table}
	
\begin{table}[H]
\centering
{
\def\arraystretch{1.25}
$ \begin{array}{c|ccccccc}
k\ \backslash\ p & 0 & 1 & 2 & 3 & 4 & 5 & 6\\ \midrule
0 & 00111111 & 10011111 & 11001111 & 11100111 & 11110011 & 11111001 & 11111100\\
1 & 01011111 & 10101111 & 11010111 & 11101011 & 11110101 & 11111010 & \\
2 & 01101111 & \hyperref[entry:P12]{\textcolor{red}{10110111}}\label{entry:F12} & 11011011 & 11101101 & 11110110 & & \\
3 & 01110111 & 10111011 & 11011101 & 11101110 & & & \\
4 & 01111011 & 10111101 & 11011110 & & & & \\
5 & 01111101 & 10111110 & & & & & \\
6 & 01111110 & & & & & &
\end{array} $
}
\caption{The set $\bigcup\limits_{p=0}^{6}1^{p}0\mathcal{T}_{7-p}$.}
\label{tab:F8}
\end{table}
	
\end{example}

\bigskip
\hrule
\bigskip

\noindent 2010 {\it Mathematics Subject Classification}: Primary 57M25;
Secondary 52C30.

\noindent \emph{Keywords:}
foil knot, state enumeration, line arrangement, lazy caterer's sequence.

\bigskip
\hrule
\end{document}